\documentclass[12pt, reqno]{amsart}
\usepackage{amsmath, amsthm, amscd, amsfonts, amssymb, mathtools, color, hyperref}
\usepackage{geometry}
\usepackage[dvipsnames]{xcolor}

\newtheorem{theorem}{Theorem}[section]
\newtheorem{lemma}[theorem]{Lemma}
\newtheorem{proposition}[theorem]{Proposition}
\newtheorem{corollary}[theorem]{Corollary}

\theoremstyle{definition}
\newtheorem{definition}[theorem]{Definition}
\newtheorem{example}[theorem]{Example}
\theoremstyle{remark}
\newtheorem{remark}[theorem]{Remark}
\numberwithin{equation}{section}

\title[Matrix polynomials and joint spectral radius] {On matrix polynomials and the joint spectral radius over max-algebras}

\author[Askar]{Askar Ali M$^{1}$}
\address{$^{1}$ Azim Premji University, Bhopal, India}
\email{askar.m@apu.edu.in}
\author[Sachindranath]{Sachindranath Jayaraman$^{2,\ast}$}
\address{$^{2}$ School of Mathematics, IISER Thiruvananthapuram, India}
\email{sachindranathj@iisertvm.ac.in, sachindranathj@gmail.com}
\thanks{$^\ast$Corresponding author}

\subjclass[2020]{Primary:  15A80 , 11B83, 15A18   Secondary: 37C25 }
\keywords{Max-algebras; matrix polynomial; joint spectral radius; matrix norm; periodic points; common eigenvectors}

\begin{document}

\noindent

\begin{abstract}
Our aim is to study matrix polynomials over max-algebras and their growth in terms of max-induced seminorms. 
In particular, we compare the set growth of a bounded family $\Psi$ of matrix polynomials, measured in terms of the 
seminorms $\eta_{\|\cdot\|}$ and $\hat{\eta}_{\|\cdot\|}$ with the induced joint spectral radius of the coefficient pool 
$\Psi_0$ of the matrix polynomials. Dynamics of max-linear maps and convergence to periodic points under a single 
joint spectral radius condition and the existence of common max-eigenvectors of the coefficient pool are also brought out.
\end{abstract}

\maketitle

\section{Introduction}

We work throughout over the field $\mathbb{R}$ of real numbers. By a max-algebra, we mean the triple $(\mathbb{R}_+,\oplus,\otimes)$, where $\mathbb{R}_+$ is the set of all nonnegative real numbers with the binary 
operations $a \oplus b = \max\{a,b\}$ and $a \otimes b = ab$. Max-algebras provide an idempotent linear-algebraic 
language for a broad class of nonlinear problems arising in discrete-event systems, scheduling, optimization and performance evaluation. This setting is isomorphic (through the exponential map) to the max-plus model commonly used in discrete-event 
dynamics. Excellent references on this subject are the monographs \cite{Butkovic, Economics-note, Max@work}. 

We shall denote by $M_n(\mathbb{R}_+)$ the collection of all $n \times n$ matrices with entries from the max-algebra 
described in the previous paragraph. Given any two such matrices $A$ and $B$, their matrix product (denoted by $AB$) 
is defined by $(A B)_{ij} = \displaystyle \max_{k} (a_{ik} .  b_{kj})$ (this also includes the product $Ax$, when $x$ is a 
vector). We shall also denote by $GL_n(\mathbb{R}_+)$ the set of all invertible matrices in a max algebra. It is easy to 
prove that an element of $GL_n(\mathbb{R}_+)$ is necessarily a generalized permutation matrix - one that is a product 
of a diagonal matrix and a permutation matrix.

Spectral analysis of max-algebra matrices plays a key role in several scenarios, such as discrete event dynamical systems, 
stability analysis, scheduling problems and so on. We refer the readers to \cite{Butkovic, Gursoy-Mason, Neda-Ghasemi, DEDS-dynamic-product, Lur, Lur2005, Muller-Peperko}. Beyond single matrices, many problems lead naturally to matrix polynomials over max-algebras. Given matrices $A_0,\dots,A_{m-1} \in M_n(\mathbb{R}_+)$, a max-matrix polynomial 
is an expression of the form $P(\lambda):= A_0 \oplus \lambda A_1 \oplus \cdots \oplus \lambda^{m-1}A_{m-1},\ 
\lambda\in \mathbb{R}_+$. Spectral properties for such polynomials have been investigated recently in 
\cite{Gursoy-Mason, Muller-Peperko-0} and \cite{Neda-Ghasemi}. The asymptotic growth of a product of matrices is 
measured by the joint spectral radius (JSR). Rota and Strang laid the foundation to this approach (see \cite{Rota-Strang}) 
and is now a standard reference in areas such as switching systems, wavelets and control \cite{Heil-Strang}. In \cite{Lur}, 
Lur initiated a study of max-algebraic version of the generalized or joint spectral radius theory. For a given norm 
$||.||$ in $\mathbb{R}^n$ and $A\in M_n(\mathbb{R}_+)$, a seminorm, as introduced by Lur, is defined by 
$$\eta_{\|\cdot\|}(A):= \displaystyle \sup \Bigl\{\frac{||Ax||}{||x||}: 0 \neq x \in \mathbb{R}_+ \Bigr\}.$$

For a bounded family of matrices $\Sigma$, one can then define (as done in \cite{Lur}) the joint spectral radius built 
from $\eta$ by 
$$\rho_\eta(\Sigma):= \displaystyle \limsup_{k\to\infty}\Big(\ \sup_{A_1,\dots,A_k \in \Sigma} 
\eta_{\|\cdot\|}(A_k  \dots  A_1)\ \Big)^{1/k}.$$ 
In \cite{Lur2005}, the author introduced the $\hat{\eta}$-norm for a matrix $A \in M_n (\mathbb{R}_+)$ by 
$$\hat{\eta}_{\|\cdot \|}(A) = \displaystyle \limsup_{k\to\infty} (\eta_{\|\cdot \|}(A^k))^{1/k}.$$ 
The set growth of a bounded family $\Sigma$ built from $\hat{\eta}$ is then given by  
$$\hat{\eta}_{\|\cdot \|}(\Sigma) = \limsup_{k\to\infty} (\sup_{A \in \Sigma}\ \eta_{\|\cdot \|}(A^k))^{1/k}.$$

This notion of set growth accounts for asymptotic stability of matrices in max-algebras. Motivated by these approaches, 
Muller and Peperko \cite{Muller-Peperko-0} developed a detailed study of the joint spectral radius in max-algebras. 
Later on, the authors in \cite{Neda-Ghasemi}, extended these notions to matrix polynomials 
with a view to study various spectral properties and norm inequalities in max-algebras. As Muller and Peperko point out, 
there is an effective calculation of the joint spectral radius in max-algebras (for bounded families), in terms of the maximum 
cycle geometric mean (see Proposition $2.1$ of \cite{Muller-Peperko-0}). We shall exploit this to illustrate our results 
by means of examples.

The purpose of this article is to connect the above notions and define the joint-spectral-radius of a family of 
max-matrix polynomials, in terms of the collection of coefficient matrices. In particular, given a bounded collection $\Psi$ of max-matrix polynomials (see Definition \ref{JSR-matrix-polynomials}), 
let us associate to it the coefficient pool $\Psi_0\subset M_n(\mathbb{R}_+)$ obtained by collecting all coefficient matrices appearing in elements of $\Psi$. Then, using the above defined max-induced seminorm $\eta_{\|\cdot\|}$,
we define a set growth rate for $\Psi$ under polynomial multiplication and compare
it to the max-algebraic joint spectral radius of $\Psi_0$. This provides a quantitative bridge between polynomial
algebra on the one hand and switched max-linear dynamics on the other hand.

We summarize our results now. We begin with a single max-matrix polynomial \(P\) and show that the joint spectral radius 
of the coefficient-set sits between the norms of the polynomial built from \(\hat{\eta}\) and \(\eta\) respectively (Theorem \ref{thm:etacap-rho-eta-inequality}). We then extend this idea for a bounded family \(\Psi\) of matrix polynomials, and we 
prove two-sided bounds comparing the set growth of \(\Psi\) with the joint spectral radius of the coefficient pool \(\Psi_0\), 
up to a factor depending only on the maximal degree in \(\Psi\) (Theorem \ref{thm:rho-eta-rho-inequality}). These inequalities sharpen the idea that the asymptotic growth of polynomial products is controlled by the growth of the matrices appearing as coefficients. When the coefficient pool \(\Psi_0\) is simultaneously triangularizable in the max-algebraic sense
(through a common permutation similarity), we derive an explicit formula for the joint spectral radius of \(\Psi_0\) as the maximum of the joint spectral radii of the induced one-dimensional diagonal pools (Theorem \ref{thm:SJR-triangular-family}). These results extend and complement the spectral descriptions of max-matrix polynomials in
\cite{Gursoy-Mason, Neda-Ghasemi} by linking them to the joint-growth quantities.

In Section \ref{Sec:com-eigenvector}, we examine common eigenvectors and periodic points of a bounded family of matrices 
(or correspondingly, polynomials with coefficients from this bounded family), and its connection with the joint spectral radius. Periodic points arise naturally in the dynamics of cone-preserving maps, and in particular in the iteration of a nonnegative 
matrix acting on $\mathbb{R}^n_+$. A classical consequence of nonlinear Perron--Frobenius theory (see for instance 
Theorem $B.4.7$ of \cite{Lemmens-Nussbaum}) states that if a nonnegative matrix $A$ has spectral radius at most one, 
then along a suitable arithmetic subsequence of iterates, every bounded orbit converges to a periodic point: there exists 
$q\geq 1$ such that for each $x$ with $\{A^k x\}_{k\ge 1}$ bounded,
$\lim_{k\to\infty}A^{kq}x=\xi_x$, where $\xi_x$ is a periodic point of $A$ and the period of $\xi_x$ dividing $q$. 
This viewpoint motivates studying convergence or limit sets through periodic points rather than
through pointwise convergence of $A^k x$. In applications, one is often led not to a single matrix, but to products drawn from 
a finite collection $\{A_1,\dots,A_N\}$, indexed by a word $\omega$ (finite or infinite). In the Euclidean (standard-algebra)  setting, S. Jayaraman {\it et. al.} \cite{Extractra-2022} prove that if each generator has spectral radius $1$ and the collection admits a nontrivial set of common eigenvectors $E$, then for each finite word $\omega$ and each initial condition
$x$ in the linear span $\mathcal{LC}(E)$, there exists an integer $q_\omega\ge 1$ such that
$\lim_{k\to\infty}A_\omega^{kq}x=\xi(x,\omega)$, where $\xi(x,\omega)$ is a periodic point
of the word product $A_\omega$ (with period dividing $q$). Moreover, once the word contains
all generators (in the sense that each symbol appears at least once), the integer $q_\omega$
and the limiting periodic point become independent of the particular choice of such a word. Recent work in the max-algebra 
setting extends periodic-point phenomena from a single matrix to words under generator-wise cycle-mean bounds \cite{DEDS-dynamic-product}. In the present paper we strengthen this by replacing individual spectral constraints on the generators with a single global stability condition expressed in terms of the max-algebraic joint spectral radius. Assuming that a finite set \(\Sigma\subset M_n(\mathbb{R}_+)\) has max-algebraic joint spectral radius at most one and 
admits common max-eigenvectors associated to its max-eigenvalues, we prove convergence of iterates along words containing 
each generator at least once and the limit being a periodic point (Theorem \ref{thm:JSR-common-eig-periodic}). As a 
consequence, we obtain fixed points for polynomial evaluations \(x\mapsto P(1)\otimes x\) when the coefficients of \(P\) 
belong to \(\Sigma\) (Corollary \ref{cor:com-eigenvector-polynomial}) as well as asymptotic stability (convergence to \(0\)) 
when the joint spectral radius is strictly less than one (Corollary \ref{cor:JSR-strict-to-zero}). These statements connect 
algebraic bounds of Section $3$ with the asymptotic product dynamics studied in \cite{DEDS-dynamic-product}.

\section{Preliminaries}

We shall collect some of the basic definitions and preliminary notions needed further. These may be found in 
\cite{Butkovic} and \cite{Gursoy-Mason}. We begin with the notion of a digraph associated to $A \in M_n(\mathbb{R}_+)$.

\begin{definition}\label{def:digraph}
For $A=(a_{ij})\in M_n(\mathbb{R}_+)$, the directed graph $G_A$ associated to $A$ has vertices $\{1,\dots,n\}$ with 
an edge $i \to j$ of weight $a_{ij}$ if and only if $a_{ij} > 0$. 
\end{definition}

A (simple) circuit $C$ has weight $w(C) = \prod a_{i_\ell i_{\ell+1}}$; the cycle mean and the 
maximum cycle mean of $A$ are then given by $\mu(C) = w(C)^{1/|C|}$ and 
$\mu(A) = \max\{\mu(C):C\subseteq G_A\}$  respectively. The maximum cycle mean $\mu(A)$ (also known as the 
Perron root) plays the role of the spectral radius in max-algebras. The (right) max–eigenpair satisfies 
$Ax = \lambda x$ with $\lambda \in \mathbb{R}_+$ being a max–eigenvalue. It is known that $\mu(A)$ 
is the largest max-eigenvalue of $A$. Moreover, if $A$ is irreducible, then $\mu(A)$ is the unique max-eigenvalue and 
every max-eigenvector is positive (see for instance Theorem $2$ of \cite{Bapat}). Moreover, the max version of Gelfand 
formula holds; that is, $\mu(A) = \displaystyle \lim_{m \to \infty} ||A^m||^{1/m}$ for any arbitrary matrix norm $||.||$ 
on $M_n(\mathbb{R})$. Before proceeding further, we wish to point out that certain functional inequalities involving the maximum cycle mean were established by Elsner and Hershkowitz \cite{Elsner-Hershkowitz}.

Let us now briefly point out the Frobenius normal form of a matrix. We will have an occasion to use it in a later section.

\begin{definition}\label{Frobenius-normal-form}
A matrix $A\in M_n(\mathbb{R}_+)$ is in \emph{Frobenius normal form} (FNF) if there exists $P \in GL_n(\mathbb{R}_+)$ 
such that
	\[
	P^{-1}AP=
	\begin{bmatrix}
		A_{11} & A_{12} & \cdots & A_{1t}\\
		0      & A_{22} & \cdots & A_{2t}\\
		\vdots & \ddots & \ddots & \vdots\\
		0      & \cdots & 0      & A_{tt}
	\end{bmatrix},
	\]
where each diagonal block $A_{ii}$ is irreducible. The blocks correspond to the strongly connected components of $G_A$. 
When every block is $1\times 1$, this reduces to an upper triangular form.
\end{definition}

Any reducible matrix can be written as a block upper triangular matrix via a permutation matrix, as above 
\cite{Bapat-Raghavan}. The Frobenius normal form is a max-algebraic analogue of block triangularization in classical 
nonnegative matrix theory. This decomposition plays a fundamental role in the spectral theory of max-algebraic matrices, 
as the Perron roots and eigenvectors can be analyzed blockwise through the Frobenius structure (see for instance 
Theorem $4.8$ of \cite{Katz-Schneider-Sergeev}). 

\begin{definition}\label{def:triangular}
 $A \in M_n(\mathbb{R}^n_+)$ is said to be triangularizable if  $P^{-1} A P$ is upper triangular for some 
$P \in GL_n(\mathbb{R}_+)$ (equivalently, for some permutation matrix $P$).  A family $\{A_1, \dots, A_n\}$ is said 
to be simultaneously triangularizable if there exists a $P \in GL_n(\mathbb{R}_+)$ such that $P^{-1}A_iP$ (once again, 
for some permutation matrix $P$) is upper triangular for all $i = 1, \dots, n$.
\end{definition}

In a recent work \cite{Askar-Jayaraman-Himadri}, we have obtained several interesting results on simultaneous 
triangularization of a family in the max-algebraic setting, some of which we shall make use of in a later section. The reader 
can also refer to \cite{Butkovic-Schneider-Sergeev}, for a detailed study of weak stability of matrices in max-algebras, 
which is connected with the structure of the Frobenius normal form. These are closely related to the spectral and structural properties of max-algebraic matrices.

\begin{definition}
For a fixed $m\in\mathbb{N}$, a max–matrix polynomial of degree $m-1$ is an expression of the form 
$P(\lambda) = A_0 \oplus \lambda A_1 \oplus \dots \oplus \lambda^{m-1}A_{m-1}, \ A_j\in M_n(\mathbb{R}_+), \  \lambda\in\mathbb{R}_+$.
\end{definition}

The (right) max–spectrum $\sigma_m[P]$ of $P(\lambda)$ consists of $k \in \mathbb{R}_+$ for which 
$P(k) v = k^m v$ with $v \neq 0$. As in the classical set up, one can associate a block companion matrix $C_P$ 
of order $mn$ to $P(\lambda)$ as follows: 
$C_P=\begin{bmatrix}
0&I&\cdots&0&0\\
&\ddots&\ddots&\vdots&\vdots\\
&&\ddots&I&0\\
A_0&A_1&\cdots&A_{m-2}&A_{m-1}
\end{bmatrix}$. 

One can then easily verify that $\sigma_m[P] = \sigma_m(C_P)$ (see for instance \cite{Neda-Ghasemi}). The following 
remark is worth pointing out and we shall have an occasion to make use of it later on.

\begin{remark}\label{rem:irreducible-Cp}
For a matrix polynomial $P(\lambda)$, assume that each of the coefficient matrices $A_j$ are irreducible. Then, the 
corresponding block companion matrix $C_P$ is also irreducible. 
\end{remark}

The following definitions and discussions are from \cite{Neda-Ghasemi}. We have stated these here for the sake of 
completeness. 

\smallskip

A vector norm $||.||$ on $\mathbb{R}^n$ is monotone if $0 \leq x \leq y \Rightarrow \|x\| \leq \|y\|$. 
Given a monotone norm $\|\cdot\|$ and $A\in M_n(\mathbb{R}_+)$, one can define a seminorm (see for instance 
\cite{Lur}) on $M_n(\mathbb{R}_+)$ by 
$$\eta_{||.||}(A):= \displaystyle \sup \biggl\{\frac{||Ax||}{||x||}: 0 \neq x \in \mathbb{R}^n_+ \biggr\}$$ 

This seminorm is subadditive with respect to $\oplus$ and submultiplicative with respect to $\otimes$; that is, 
$\eta_{\|\cdot\|}(X\oplus Y) \leq \eta_{\|\cdot\|}(X) + \eta_{\|\cdot\|}(Y), \ 
\eta_{\|\cdot\|}(XY) \leq \eta_{\|\cdot\|}(X)\eta_{\|\cdot\|}(Y)$.

\begin{definition}\label{Eta-norm of a polynomial}
For a polynomial \(P(\lambda ) = \displaystyle \sum_{j=0}^{m-1}\lambda^j A_j\), 
\[
\eta_{\|\cdot\|}[P]:= \displaystyle \max_{0\leq j\leq m-1}\eta_{\|\cdot\|}(A_j).
\]
\end{definition}

For a polynomial $P(\lambda) = A_0 \oplus \lambda A_1 \oplus \cdots \oplus \lambda^{m-1}A_{m-1}$, we denote by
$\Sigma_P = \{A_0, A_1, \cdots, A_{m-1}\}$, the collection of all coefficient matrices.  With the max-algebra operations, there 
is an obvious way to define the sum and product of two max-matrix polynomials. It is clear that the set $V$ of all max-matrix 
polynomials is a real vector space under $\oplus$. Given any norm $||.||$ on $\mathbb{R}^n$ and any $P(\lambda) = 
\displaystyle \sum_{j=0}^{m-1} \lambda^{j} A_j \in V$, one can define 
$$||P(\lambda)||:= \max\{||A_0||, \cdots, ||A_{m-1}||\},$$ where $||A_j||$ is the induced matrix norm from the 
vector norm $||.||$ on $\mathbb{R}^n$. It is easy to verify that $V$ is a real normed linear space with respect to the 
above norm and that a subset $\Psi \subset V$ is bounded if and only if there exists an $M > 0$ such that 
$||P(\lambda)|| \leq M$ for all $P(\lambda) \in \Psi$. It also turns out as a consequence of Proposition $2.7$ of 
\cite{Neda-Ghasemi} that $\Psi$ is bounded if and only if there is an $M > 0$ such that $\eta_{\|\cdot\|}[P] \leq M$ 
for all $P(\lambda) \in \Psi$. Before proceeding further, let us mention that we shall denote by $\Psi^{k}$, the set of 
all $k$-fold products of $P(\lambda)$ from the set $\Psi$. Similarly, we shall denote by $\Sigma_P^{k}$, the $k$-fold product 
of coefficient matrices coming from the set $\Sigma_P$, corresponding to a matrix polynomial $P(\lambda)$.

\begin{definition}\label{JSR-matrix-polynomials}
Let \(\Psi\) be a bounded set of max–matrix polynomials.
Define the set growth of \(\Psi\) by
\begin{equation}\label{Eta-norm-Psi}
\eta_{\|\cdot\|}(\Psi):=\displaystyle \limsup_{k\to\infty}\Big(\ \sup_{P\in\Psi^{k}} \eta_{\|\cdot\|}[P]\ \Big)^{1/k}.
\end{equation}

If \(\Psi_0:= \displaystyle \bigcup_{P\in\Psi} \Sigma_P\) is the coefficient pool, the  joint spectral radius built from 
\(\eta\) is defined by 
\begin{equation}\label{Rho-eta-Psi_0}
\rho_\eta(\Psi_0):= \displaystyle \limsup_{k\to\infty}\Big(\sup_{A_1,\dots,A_k\in\Psi_0} \eta_{\|\cdot\|}(A_1 \cdots  A_k)\Big)^{1/k}.
\end{equation}

\end{definition}

The above definition of $\rho_\eta(.)$ is the same as the definition of the joint spectral radius of a collection of matrices, 
except we have used the $\eta_{||.||}$.

\begin{definition}\label{def:eta-hat-norm}
For a matrix $A \in M_n (\mathbb{R}_+)$, the $\hat{\eta}$-norm is defined as 
\begin{equation}\label{Eta-hat-norm-matrix}
\hat{\eta}_{\|\cdot \|}(A):= \displaystyle \limsup_{k\to\infty} (\eta_{\|\cdot \|}(A^k))^{1/k}.
\end{equation}

For a polynomial $P(\lambda) = A_0 \oplus \lambda A_1 \oplus \cdots \oplus \lambda^{m-1}A_{m-1}$, 
\begin{equation}\label{Eta-hat-norm-polynomial}
\hat{\eta}_{\|\cdot \|}(P) = \max_j\ \hat{\eta}_{\|\cdot \|}(A_j). 
\end{equation}
\end{definition}

Finally, the set growth of a bounded $\Psi$ built from \(\hat{\eta}\) is 
\begin{equation}\label{Set-growth-from-eta-hat}
\hat{\eta}_{\|\cdot \|}(\Psi) = \limsup_{k\to\infty} (\sup_{P \in \Psi}\ \eta_{\|\cdot \|}(P^k))^{1/k}.
\end{equation}

Our objective is to compare the set growth of a bounded family $\Psi$ of matrix polynomials, measured in terms of 
the seminorms $\eta_{\|\cdot\|}$ and $\hat{\eta}_{\|\cdot\|}$ as defined in Equations \eqref{Eta-norm-Psi} and  \eqref{Set-growth-from-eta-hat} respectively, with the joint spectral radius of the coefficient pool $\Psi_0$. Since the set 
growth is defined in terms of $\eta_{\|\cdot\|}$, it is natural to define the joint spectral radius of $\Psi_0$ via the same 
seminorm, as in Equation \eqref{Rho-eta-Psi_0}, so that the two quantities are directly comparable. We note that this choice 
of norm does not affect the value of the joint spectral radius, since all monotone norms on $\mathbb{R}^n$ are equivalent 
in finite dimensions.

The following discussion is worth pointing out. Given a bounded subset $\Sigma$ of $M_n(\mathbb{R}_+)$, 
the max-algebraic version of the generalized spectral radius $\mu(\Sigma$) of $\Sigma$, is defined as 
$$\mu(\Sigma) = \displaystyle \limsup_{m \to \infty} \big[\sup_{A \in \Sigma^{m}} \mu(A) \big]^{1/m},$$ and is equal 
to the joint spectral radius of $\Sigma$, 
$$\rho(\Sigma) = \displaystyle \limsup_{m \to \infty} \big[\sup_{A \in \Sigma^{m}} ||A|| \big]^{1/m}.$$ We state this as 
a theorem below and will have an occasion to use it later on.

\begin{theorem}\label{Determining the JSR}
Let $\Sigma \subset M_n(\mathbb{R}_+)$ be a bounded set. Then, the following hold.
\begin{enumerate}
\item (Proposition $2.1$, \cite{Muller-Peperko-0}) $\mu(\Sigma) = \mu(S(\Sigma))$, where $S(\Sigma) = 
\displaystyle \bigoplus_{A \in \Sigma} A$.
\item (Corollary $2.3$, \cite{Muller-Peperko-0}) $\mu(\Sigma) = \rho(\Sigma)$.
\end{enumerate}
\end{theorem}

Statement $2$ in the above theorem is known as the Berger-Wang formula. Before proceeding with the main results, 
we bring out to the attention of readers, the following, although we do not make explicit use of these in the present work. 
Peperko \cite{Peperko} obtained inequalities concerning the generalized spectral radius \ \& \ its max-version and also 
gave a description of the maximum cycle mean of a max-algebraic matrix (see Definition $2.3$ and Example $2.5$ of \cite{Peperko}). In a much recent work, Bogdanovic and Peperko \cite{Bogdanovic-Peperko} define as well as establish 
several monotonicity properties of the generalized spectral radius and the joint spectral radius in a more general setting. 
We also refer the reader to \cite{ Thaghizadeh et al - 1, Thaghizadeh et al - 2}, where the max numerical range $W_{\max}(\Sigma)$ of a bounded set $\Sigma$ of nonnegative matrices is studied in connection with the generalized 
spectral radius.

\section{Main Results}

The main results are presented in this section. We subdivide this section into three subsections for ease of reading. We begin 
with a subsection, where we prove some norm inequalities connected to the joint spectral radius.

\subsection{Joint spectral radius and norm inequalities}\label{Sec:JSR-norm-inequalities}\hspace*{0.5cm}

We bring out norm-inequalities for a polynomial and a bounded collection of polynomials that connects 
the joint spectral radius and the set growth built from $\eta$ and $\hat{\eta}$ in this section. Our first result is the following.

\begin{theorem}\label{thm:etacap-rho-eta-inequality}
Let $P(\lambda) = A_0 \oplus \lambda A_1 \oplus \cdots \oplus \lambda^{m-1}A_{m-1}$ be a matrix polynomial. 
Then, $\hat{\eta}_{\| \cdot \|}(P) \leq \rho_\eta(\Sigma_P) \leq \eta_{\| \cdot \|} (P)$.
\end{theorem}

\begin{proof}
By Lemma $1 (V)$ of \cite{Lur}, we have, $\eta_{\| \cdot \|}  (A B) \leq \eta_{\| \cdot \|} (A) \cdot \eta_{\| \cdot \|} (B)$. 
Then,
\begin{equation*}
\eta_{\| \cdot \|} (A_{i_1} A_{i_2} \cdots A_{i_k}) \leq 
\displaystyle \prod_{j=1}^{k} \eta_{\| \cdot \|} (A_{i_j}) \leq ( \max_j \eta_{\| \cdot \|} (A_{i_j})) = 
(\eta_{\| \cdot \|} (P))^k .
\end{equation*}
By taking the $k^{th}$ root and then the supremum over $\Sigma_P^{k}$, 
we get 
\[\rho_\eta(\Sigma_P) \leq \eta_{\| \cdot \|} (P).\] 
To prove the lower bound, observe that for each $j$,
\begin{equation*}
\begin{aligned}
    \rho_\eta (\Sigma_P) & = \displaystyle \limsup_{k \to \infty} (\sup_{\Sigma_P} 
    ( \eta_{\| \cdot \|} (A_{i_1}  A_{i_2} \cdots A_{i_k})))^{1/k}\\
    &\geq \limsup_{k \to \infty} ( \eta_{\| \cdot \|} (A_{i_j}^k )^{1/k}\\
    & = \hat{\eta}_{\| \cdot \|}(A_{i_j}).
\end{aligned}
\end{equation*}
\end{proof}

We now move on to bounded families of matrix polynomials.

\begin{theorem}\label{thm:rho-eta-rho-inequality}
Let $\Psi$ be a bounded family of matrix polynomials. Then, 
    \[\rho_{\eta} (\Psi_0) \leq \eta_{\| \cdot \|} (\Psi) \leq m \rho_{\eta} (\Psi_0),\] 
where $m = \displaystyle \max_{P \in \Psi} \deg(P)$.
\end{theorem}

\begin{proof}
Let $P_i(\lambda)= A_{0}^{(i)} \oplus \lambda A_{1}^{(i)} \oplus \cdots \oplus \lambda^{m_i-1}A_{{{m_i-1}}}^{(i)}$. 
It is easy to verify that 
\[P_1(\lambda)  P_2(\lambda) \cdots P_k (\lambda) = \displaystyle 
\bigoplus_{ s=0}^{m_1 +\cdots +m_k} \lambda^s C_s,\] where 
$C_s = \displaystyle \bigoplus_{i_1+ \cdots + i_k=s, 0\leq i_j \leq m-1} 
(A_{i_1}^{(1)} A_{i_2}^{(2)} \cdots A_{i_k}^{(k)})$. Then, one of the coefficients $C_s$ contains 
the term $A_1 A_2 \cdots A_k$. For this coefficient, 
 \begin{equation*}
    \begin{aligned}
    C_s &\geq A_1 A_2 \cdots A_k.\\
    \text{Thus, } \eta_{\|.\|}(C_s) &\geq \eta_{\|.\|}(A_1 A_2 \cdots A_k).
    \end{aligned}
\end{equation*}
We also have, 
\[\eta_{\|.\|} (P_1 P_2 \cdots P_k) = \max_s \eta_{\|.\|}(C_s) \geq \eta_{\|.\|} (A_1 A_2 \cdots A_k).\]
By taking the supremum over $P_i \in \Psi$, and $A_i \in \Psi_0$, the above inequality becomes,
\[\sup_{P \in \Psi^{k}} \eta_{\|.\|}(P) \geq 
\sup_{\Psi_0^{k}} (A_1 A_2 \cdots A_k).\]
By taking the $k^{th}$ root and then taking the $\limsup$, we get 
\[\eta_{\|.\|}(\Psi) \geq \rho_\eta (\Psi_0).\] 
To get the upper bound, by sub-additivity of the $\eta-$semi norm, we have 
\begin{equation*}
    \begin{aligned}
    \eta_{\|.\|}(C_s) 
    & \leq \displaystyle \bigoplus_{i_1+ \cdots + i_k=s, 0\leq i_j \leq m-1} \eta_{\|.\|}  
    (A_{i_1}^{(1)} A_{i_2}^{(2)} \cdots A_{i_k}^{(k)})\\
    &\leq |\{ i_1+ \cdots + i_k=s|\ 0\leq i_j \leq m-1\}| \sup \eta_{\|.\|} 
    (A_{i_1}^{(1)} A_{i_2}^{(2)} \cdots A_{i_k}^{(k)})\\
    &\leq m^k \sup \eta_{\|.\|}  (A_{i_1}^{(1)} A_{i_2}^{(2)} \cdots A_{i_k}^{(k)}).
    \end{aligned}
\end{equation*}
From the above, we have 
$\eta_{\|.\|} (P_1 \otimes P_2 \cdots P_k) \leq m^k \sup \eta_{\|.\|}(A_{i_1}^{(1)} A_{i_2}^{(2)} \cdots A_{i_k}^{(k)})$. 
By taking the $k^{th}$ root, the supremum over $\Psi^{k}$, and then 
the $\limsup$, we get 
\[\eta_{\|.\|}(\Psi) \leq m\ \rho_\eta (\Psi_0).\]
We thus have $\rho_{\eta} (\Psi_0) \leq \eta_{\| \cdot \|} (\Psi) \leq m \rho_{\eta} (\Psi_0)$.
\end{proof}

\subsection{Triangularizable coefficient matrices}\hspace*{0.5cm}

We now consider the case when the coefficients of the matrix polynomial are simultaneously triangularizable. 

\begin{definition}\label{def:perm-invariant-norm}
A norm $\| \cdot \|$ on $\mathbb{R}_+^n$ is called permutation invariant if $\| Px \| = \| x \|$, for any 
$x \in \mathbb{R}_+^n$, and for any permutation matrix $P$.
\end{definition}

\begin{lemma}\label{lem:perm-norm}
If $\| \cdot \|$ is a permutation invariant norm, then the $\eta_{\| \cdot \|}$ induced from $\| \cdot \|$ satisfies the 
following property. For any $A \in M_n(\mathbb{R}_+)$ and any permutation matrix $P$, 
$\eta_{\| \cdot \|}(P^{-1}AP) = \eta_{\| \cdot \|}(A)$.
\end{lemma}

\begin{proof}
Observe that $\eta_{\| \cdot \|}(P^{-1}AP) = \sup_{x \neq 0} \frac{\|P^{-1}APx\|}{\|x\|}$. For $y=Px$, we have the following.
\begin{equation*}
    \begin{aligned}
    \eta_{\| \cdot \|}(P^{-1}AP) 
    & = \sup_{x \neq 0} = \frac{\|P^{-1}Ay\|}{\|P^{-1}y\|}\\
    & = \sup_{x \neq 0} = \frac{\|Ay\|}{\|y\|}\\
    & = \eta_{\|\cdot\|} (A).
    \end{aligned}
    \end{equation*}
\end{proof}

Let us now consider a matrix polynomial where the coefficient matrices are simultaneously triangularizable. 

\begin{proposition}\label{prop:triangular-spectrum}
Let $P(\lambda) = A_0 \oplus \lambda A_1 \oplus \cdots \oplus \lambda^{m-1}A_{m-1}$ be a matrix polynomial, with 
$\Sigma_P = \{A_0,A_1, \cdots, A_{m-1}\}$ being a simultaneously triangularizable family through a permutation matrix $S$.  
Taking $A_i' = S^{-1}AS$ to be upper triangular, define the scalar polynomials 
    \[p_i(\lambda) = (A'_0)_{ii} \oplus \lambda (A'_1)_{ii}\oplus \cdots \oplus \lambda^{m-1}(A'_{m-1})_{ii}. \] 
Then, $\sigma_m[P] = \displaystyle \bigcup_{i=1}^{n} \sigma_m[p_i]$.
\end{proposition}

\begin{proof}
By definition, \(k \in \sigma_m[P]\) if and only if there exists \(v\geq 0\), \(v\neq 0\), such that
\[
P(k) v = k^{m} v.\] 
Similarly, for the scalar polynomial \(p_i\), we have \(k\in\sigma_m[p_i]\) if and only if \(p_i(k)=k^m\). 
Set \(P'(\lambda):=S^{-1}P(\lambda)S = \displaystyle \bigoplus_{r=0}^{m-1} \lambda^r A'_r\). 
Since each \(A'_r\) is upper triangular, so is
\[
P'(k)=A'_0 \oplus kA'_1 \oplus \cdots \oplus k^{m-1}A'_{m-1}. \] 
The $i^{th}$ diagonal entry of \(P'(k)\) is exactly \(p_i(k)\). An upper triangular matrix is already a block
upper triangular matrix with $1 \times 1$ irreducible diagonal blocks, and hence is in Frobenius normal form (FNF). 
These blocks are precisely the scalars \(p_i(k)\) on the diagonal. Then, by Theorem $4.5.4$ of \cite{Butkovic}, it follows 
that the eigenvalues of all irreducible blocks are eigenvalues of the matrix $P'$ and hence of \(P\) by similarity. We thus 
have $\sigma_m[P] = \displaystyle \bigcup_{i=1}^{n} \sigma_m[p_i]$.
\end{proof}

\begin{definition}\label{def:SJR-diagonal-entry}
Given a bounded family $\Psi$ of matrix polynomials such that the set $\Psi_0$ of all coefficient matrices is a simultaneously triangularizable family through a permutation matrix $S$, let $\Delta\Psi_0:= \{S^{-1}AS\ \mid A \in \Psi_0\}$ and 
$\Delta\Psi_0^{(i)} := \{ (S^{-1}AS)_{ii}\ \mid A \in \Psi_0\}$. For this scalar pool, we define the joint spectral radius by 
\[\rho_\eta (\Delta\Psi_0^{(i)}) = 
\displaystyle \limsup_{k \to \infty} \big(\sup_{a_1,a_2,\cdots,a_k \in  \Delta\Psi_0^{(i)}} a_1a_2\cdots a_k \big)^{1/k}. \]
\end{definition}

\begin{lemma}\label{lem:limsup}
Let $W \subset \mathbb{R}_+$ be a bounded set. Then, 
\[\displaystyle \limsup_{k \to \infty} (\sup_{a_1,a_2, \cdots, a_k \in  \Delta\Psi_0^{(i)}} a_1a_2 \cdots a_k)^{1/k} = 
\displaystyle \sup_{a \in W}\ a.\]
\end{lemma}

\begin{proof}
Since $W \subset \mathbb{R}_+$ is bounded and nonempty, we may set $M := \displaystyle \sup_{a \in W} a \in (0,\infty)$. 
For each $k \in \mathbb{N}$, define $P_k := \displaystyle \sup_{a_1,\dots,a_k \in W} \bigl(a_1 a_2 \cdots a_k\bigr)^{1/k}$. 
We shall show that $\displaystyle \limsup_{k \to \infty} P_k = M$. To prove this, let us fix $k \in \mathbb{N}$ and choose arbitrary $a_1,\dots, a_k \in W$. By the very definition of $M$, we have $a_j \leq M$ for each $j=1,\dots,k$, so that 
\[
    a_1 a_2 \cdots a_k \;\le\; M^k,
    \qquad\text{hence}\qquad
    \bigl(a_1 a_2 \cdots a_k\bigr)^{1/k} \leq M.
\]
Taking the supremum over all $a_1,\dots, a_k \in W$ we see that $P_k \leq M \quad \text{for all } k$. Therefore
$\displaystyle \limsup_{k \to \infty} P_k \leq M$. If $\varepsilon > 0$ is arbitrary, then by the definition of $M$ there exists $a_\varepsilon \in W$ such that $a_\varepsilon > M - \varepsilon$. For each $k \in \mathbb{N}$, consider the choice 
$a_1 = a_2 = \cdots = a_k = a_\varepsilon$. Then $\displaystyle \bigl(a_1 a_2 \cdots a_k\bigr)^{1/k} 
= \bigl(a_\varepsilon^k\bigr)^{1/k} = a_\varepsilon > M - \varepsilon$. Since $P_k$ is the supremum over all such products, 
we obtain $P_k \geq a_\varepsilon > M - \varepsilon \quad \text{for all } k$. 
Hence, $\displaystyle \limsup_{k \to \infty} P_k \geq M - \varepsilon$. $\varepsilon>0$ being arbitrary, we conclude that $\displaystyle \limsup_{k \to \infty} P_k \geq M$. Thus, $\displaystyle \limsup_{k \to \infty} P_k = M = \sup_{a \in W} a$.
\end{proof}

We are now in a position to prove the main result of this section. 

\begin{theorem}\label{thm:SJR-triangular-family}
Let $\Psi$ be a bounded family of matrix polynomials with $\Psi_0$ being a simultaneously triangularizable family. Assume  further that $\|.\|$ is a permutation invariant monotone norm and $\eta$ is induced by this norm. Then, 
 $\rho_\eta(\Psi) = \displaystyle \max_i \rho_\eta(\Delta\Psi_0^{(i)})$.
\end{theorem}

\begin{proof}
Let the matrices in $\Psi_0$ be simultaneously triangularizable through a permutation matrix $S$. For each $A_i \in \Psi_0$, assume $B_i = S^{-1}A_i S$ is upper triangular. Then, using Lemma \ref{lem:perm-norm}, we have $\eta_{\|\cdot \|}(A) = \eta_{\|\cdot \|}(B) $. It therefore suffices to verify $\rho_\eta(\Psi)$ for $B_i \in \Delta\Psi_0$. Let, 
$s:= \displaystyle \max_i \sup_{B \in \Delta\Psi } B_{ii}$ and $e_i$ be the $i^{th}$ standard basis vector. Then, 
 $|(B_1 B_2 \cdots B_k) e_i \| \geq (B_1  B_2  \cdots  B_k)_{ii} \cdot \|e_i\|$. 
This gives $\eta_{\|\cdot \|}(B_1 B_2 \cdots B_k) \geq (B_1 B_2  \cdots B_k)_{ii} = \prod_{r=1}^k (B_r)_{ii}.$ 
We then have, $\rho_\eta (\Delta\Psi) = \rho_\eta (\Psi) \geq \displaystyle \sup_{B \in \Delta\Psi } B_{ii}$.
Thus, $\rho_\eta (\Psi) \geq \displaystyle \max_i \sup_{B \in \Delta\Psi } B_{ii}= s$. Since $\Psi$ is a bounded set, we have $\displaystyle \sup_{B \in \Delta\Psi } B_{ij} < M$, for some $M > 0$. Therefore, each contributing term in 
$(B_k  \cdots  B_1)_{ij}$ contains at most \(n-1\) off-diagonal factors, each bounded above by \(M\), and at least \(k-(n-1)\) diagonal factors, each bounded above by \(s\). Hence,
$(B_1 B_2 \cdots B_k)_{ij} \leq M^{n-1}s^{k-(n-1)} \leq Cs^k, \ \text{ with } C:=\max\{1,M^{n-1}s^{-(n-1)}\}$. 
Since all norms on $\mathbb{R}_+^n$ are equivalent, there exists a $c > 0$ such that 
$\eta_{\|\cdot\|} (B_1 B_2  \cdots  B_k) \leq \displaystyle c \max_{i,j} (B_1 B_2  \cdots B_k)_{ij} \leq cCs^k$.
Taking the $k^{th}$ root and the $\limsup$, the above inequality becomes, $\rho_\eta (\Psi) \leq s$. Combining these two inequalities, we have, $\rho_\eta(\Psi) = \displaystyle \max_i \rho_\eta(\Delta\Psi_0^{(i)})$.
\end{proof}

We illustrate the above theorem by means of an example.

\begin{example}\label{eg-1}
Let $\Psi = \{P_1, P_2\}$ where $P_i(\lambda) = A_0^{(i)} \oplus \lambda A_1^{(i)}$ are max-matrix polynomials over $M_3(\mathbb{R}_+)$, with coefficient matrices
	\[
	A_0^{(1)} =
	\begin{pmatrix} 
		3 & 0 & 2 \\ 
		0 & 1 & 0 \\ 
		0 & 2 & 4 
	\end{pmatrix},
	\ 
	A_1^{(1)} =
	\begin{pmatrix} 
		2 & 3 & 1 \\ 
		0 & 2 & 0 \\ 
		0 & 1 & 3 
	\end{pmatrix},
	\ 
	A_0^{(2)} =
	\begin{pmatrix} 
		1 & 2 & 3 \\ 
		0 & 0 & 0 \\ 
		0 & 1 & 2 
	\end{pmatrix},
	\ \text{and } \ 
	A_1^{(2)} =
	\begin{pmatrix} 
		4 & 1 & 0 \\ 
		0 & 1 & 0 \\ 
		0 & 3 & 2 
	\end{pmatrix}.
	\] 
The coefficient pool is $\Psi_0 = \bigl\{A_0^{(1)},\, A_1^{(1)},\, A_0^{(2)},\,A_1^{(2)}\bigr\}$. Let $S$ be the permutation 
matrix corresponding to the transposition $(2\;3)$; that is, 
	\[
	S =
	\begin{pmatrix} 
		1 & 0 & 0 \\ 
		0 & 0 & 1 \\ 
		0 & 1 & 0 
	\end{pmatrix} = S^{-1}.
	\] 
For any $A \in \Psi_0$, a direct computation gives $S^{-1}AS$ is upper triangular. Hence $\Psi_0$ is simultaneously triangularizable through $S$. Applying this to each element of $\Psi_0$, we obtain 
	\begin{align*}
		S^{-1}A_0^{(1)}S & 
		= \begin{pmatrix} 
			3 & 2 & 0 \\ 
			0 & 4 & 2 \\ 
			0 & 0 & 1 
		\end{pmatrix},
		&
		S^{-1}A_1^{(1)}S & = 
		\begin{pmatrix} 
			2 & 1 & 3 \\ 
			0 & 3 & 1 \\ 
			0 & 0 & 2 
		\end{pmatrix},
		\\[6pt]
		S^{-1}A_0^{(2)}S & = 
		\begin{pmatrix} 
			1 & 3 & 2 \\ 
			0 & 2 & 1 \\ 
			0 & 0 & 0 
		\end{pmatrix},
		&
		S^{-1}A_1^{(2)}S & = 
		\begin{pmatrix} 
			4 & 0 & 1 \\ 
			0 & 2 & 3 \\ 
			0 & 0 & 1 
		\end{pmatrix}.
	\end{align*}
	
Set $\Delta\Psi_0 = \{S^{-1}AS \mid A \in \Psi_0\}$. Reading off the diagonal entries at each position $i = 1,2,3$, the scalar 
pools are,
	\[
	\Delta\Psi_0^{(1)} = \{3,\,2,\,1,\,4\},
	\qquad
	\Delta\Psi_0^{(2)} = \{4,\,3,\,2,\,2\},
	\qquad
	\Delta\Psi_0^{(3)} = \{1,\,2,\,0,\,1\}.
	\] 
By Lemma \ref{lem:limsup}, $\rho_\eta\!\left(\Delta\Psi_0^{(1)}\right) = \sup \Delta\Psi_0^{(1)} = 4,
	\ \rho_\eta\!\left(\Delta\Psi_0^{(2)}\right) = \sup \Delta\Psi_0^{(2)} = 4,$ \\
	$\rho_\eta\!\left(\Delta\Psi_0^{(3)}\right) = \sup \Delta\Psi_0^{(3)} = 2$. 
We then infer from Theorem~\ref{thm:SJR-triangular-family} that,
	\[
	\rho_\eta(\Psi_0)
	= \max_{i=1,2,3}\,\rho_\eta\!\left(\Delta\Psi_0^{(i)}\right)
	= \max\{4,\,4,\,2\} = 4.
	\]
One can verify this directly from statements $(1)$ and $(2)$ of Theorem \ref{Determining the JSR}, which actually gives 
	$\rho_\eta(\Psi_0) = \mu (A_0^{(1)}\oplus A_1^{(1)} \oplus A_0^{(2)}\oplus A_1^{(2)})$. We also have, 
	\[ A_0^{(1)}\oplus A_1^{(1)} \oplus A_0^{(2)}\oplus
	A_1^{(2)} = \begin{pmatrix}
		4&3&3\\
		0&2&0\\
		0&3&4
	\end{pmatrix}.\]
Thus, $\mu (A_0^{(1)} \oplus A_1^{(1)} \oplus A_0^{(2)} \oplus A_1^{(2)}) = 4$, confirming $\rho_\eta(\Psi_0) = 4$.
\end{example}

\subsection{Common eigenvectors and periodic points}\label{Sec:com-eigenvector}\hspace*{0.5cm}

This is the last section of the paper where we investigate self maps on $\mathbb{R}_+^n$ of the form $Ax$, for some 
$A\in M_n(\mathbb{R}_+)$ and $x \in \mathbb{R}_+^n$. For a bounded family of matrices $\Sigma$, and matrix 
polynomials with coefficients from this family $\Sigma$, we examine the relation between common eigenvectors and 
periodic points of the above maps with the joint spectral radius. More precisely, given $A\in M_n(\mathbb{R}_+)$, we 
consider the max--linear self-map $T_A:\mathbb{R}_+^n\to\mathbb{R}_+^n,\ T_A(x)= Ax$. A vector $x\in\mathbb{R}_+^n$ is 
called a fixed point of $A$ (or of $T_A$) if $Ax = x$. More generally, $x$ is called a periodic point of $A$ if there exists
$q\in\mathbb{N}$ such that $A^{q} x = x \ (\text{equivalently } T_A^{\,q}(x) = x)$. 
The least such $q$ (when it exists) is called the period of $x$.  Before proceeding further with our results, the following 
is worth pointing out. In \cite{Butkovic-CuninghameGreen}, P. Butkovic and Cuninghame-Green discussed the periodic 
behaviour of max-algebraic matrices. In particular, for an irreducible matrix $A$, there exist integers $p\geq 1$ and 
$k_0$ such that $A^{k+p}=\lambda(A)^{p} A^k$ for all $k \geq k_0$, where $\lambda(A)$ is the max-algebraic eigenvalue (maximum cycle mean). The smallest such $p$ is called the period of $A$, and is denoted by $\mathrm{per}(A)$. A matrix 
is called robust if every orbit eventually reaches an eigenvector of $A$; equivalently, $\mathrm{per}(A)=1$. This 
single-matrix robustness condition can be seen as a precursor to our setting: where Butkovič and Cuninghame-Green require periodicity one for a single matrix to guarantee orbit convergence, we replace this matrix-wise condition with a single 
global bound in terms of the joint spectral radius of the entire collection $\Sigma$, and characterize convergence to 
periodic points for all word products simultaneously.

\begin{proposition}\label{prop:com-eigenvector-polynomial-eigenvalue}
Let $P(\lambda) = A_0 \oplus \lambda A_1 \oplus \cdots \oplus \lambda^{m-1}A_{m-1}$ be a matrix polynomial. 
Assume that there exists a nontrivial vector $v \in \mathbb{R}^n_+$ such that $A_j v = \alpha_jv, \text{ for } 
j=0,1,\cdots,m-1$. Define $q(\lambda) =  \alpha_0 \oplus \lambda \alpha_1 \oplus \cdots \oplus \lambda^{m-1}\alpha_{m-1}$. 
If $q(k) = k^m$ for some $k \geq 0$, then, $P(k)v = k^m v, \ \text{ equivalently, }\ k \in \sigma_m[P]$.
\end{proposition}

\begin{proof}
Using max–linearity and the hypothesis, we have 
$P(k) v = \displaystyle \bigoplus_{j=0}^{m-1} k^j (A_j  v) 
= \displaystyle \bigoplus_{j=0}^{m-1} (k^j \alpha_j)\, v 
= \displaystyle \big(q(k)\big)\, v 
= k^m v$.
\end{proof}

The following proposition is worth pointing out. 

\begin{proposition}\label{prop:eigen-Cp-1}
 Let $P(\lambda) = A_0\oplus \lambda A_1\oplus\cdots\oplus \lambda^{m-1}A_{m-1}$ be a matrix polynomial with all the 
 $A_j$'s irreducible. Let $C_P$ be the associated companion matrix. If there exists a nonzero periodic point of 
 $y \mapsto C_P y$, then $\mu(C_P) = 1$.
\end{proposition}

\begin{proof}
Assume that $C_P^{q} x = x$ for some $x \neq 0$ and $q \geq 1$. Define 
$y:= x \oplus\ (C_P x) \oplus\ (C_P^{2} x)\ \oplus \cdots \oplus\ (C_P^{q-1} x)$. 
By max--linearity, we have $C_P y = (C_P x) \oplus (C_P^{2} x) \oplus \cdots \oplus (C_P^{q} x) 
= (C_P x) \oplus (C_P^{2} x) \oplus \cdots \oplus x = y$, 
where the last equality comes from the fact that $C_P^{q} x = x$. Hence $y \neq 0$ is a max--eigenvector of $C_P$
with eigenvalue $1$. Since the $A_j$'s are irreducible, by Remark \ref{rem:irreducible-Cp}, $C_P$ is irreducible. 
For an irreducible max--algebraic matrix, the (unique) eigenvalue admitting a nonzero nonnegative
eigenvector equals the maximum cycle geometric mean $\mu(C_P)$. Therefore $\mu(C_P) = 1$.
\end{proof}

For a non-empty subset $S$ of $\mathbb{R}^n_+$, we denote by $\mathcal{LC}(S)$ the max--cone generated by $S$. For a bounded family of matrices $\Sigma = \{A_1,\dots,A_N\} \subset M_n(\mathbb{R}_+)$, we denote the collection of common eigenvectors of $\Sigma$ as $E$. Note that $\mathcal{LC}(E)$ is invariant for the collection $\Sigma$, as well 
as for any matrix polynomial $P(\lambda)$ with coefficients from the set $\Sigma$. For a $p-$letter word 
$\omega = \omega_1 \omega_2 \cdots \omega_p$ on 
$\{1,\dots,N\}$, we define the associated product by 
$A_\omega:= A_{\omega_p} A_{\omega_{p-1}} \cdots A_{\omega_1}$. The following theorem replaces the individual spectral constraints $\mu(A_i)\leq 1$ in Theorem $3.4$ of \cite{DEDS-dynamic-product} by the single joint spectral radius
condition $\rho_\eta(\Sigma)\leq 1$, which controls the growth of all products uniformly. This global bound is essential in our setting, as it allows one to deduce full convergence of $A_\omega^k  x$ for every word $\omega$ using a single stability 
parameter, rather than relying on matrix–wise spectral estimates. Our main result is the following.

\begin{theorem}\label{thm:JSR-common-eig-periodic}
Let $\Sigma = \{A_1,\dots,A_N\}\subset M_n(\mathbb{R}_+)$. Assume that $\rho_\eta(\Sigma)\leq 1$ and that $\Sigma$ 
admits a nontrivial set $E$ of common max-eigenvectors. Then for every $x \in \mathcal{LC}(E)$, and a word 
$\omega = \omega_1 \omega_2 \cdots \omega_p$ on $\{1,\dots,N\}$, the limit
$\displaystyle \lim_{k\to\infty} A_\omega^k\otimes x :=  \xi_{x,\omega}$ exists, and $\xi_{x,\omega}$ is a periodic point of $A_\omega$.
\end{theorem}

\begin{proof}
For $A_i \in \Sigma$ and $v_j \in E$, let $A_i v_j = \alpha_{ij}v_j$, for some $\alpha_{ij} \in \mathbb{R}_+$. Then 
max--linearity gives
$A_i^k  v_j = \alpha_{ij}^k v_j$ for all $k \geq 1$. Hence $\eta_{\|\cdot \|}(A_i^k)\;\geq\;\frac{\|A_i^kv_j\|}{\|v_j\|} 
= \frac{\|\alpha_{ij}^k v_j\|}{\|v_j\|} = \alpha_{ij}^k$. Therefore, $\hat{\eta}_{\|\cdot \|}(A_i) \geq \alpha_{ij}$. It then 
follows from Theorem \ref{thm:etacap-rho-eta-inequality} that $\rho_\eta(\Sigma) \geq \hat{\eta}_{\|\cdot \|}(A_i)$ and so 
$\alpha_{ij} \leq \rho_\eta(\Sigma) \leq 1$. Now fix a word $\omega = \omega_1 \cdots \omega_p$ and set 
$\beta_j(\omega):= \displaystyle \prod_{t=1}^{p} \alpha_{\omega_t,j} \in [0,1]$.  Then
$A_\omega v_j = \beta_j(\omega)v_j$ and $A_\omega^k v_j = \beta_j(\omega)^k v_j$. 
Take $x \in \mathcal{LC}(E)$ and choose scalars $\gamma_1, \dots, \gamma_r \in \mathbb{R}_+$ such that
$x = \displaystyle \bigoplus_{j=1}^{r} \gamma_j v_j$.  Using max--linearity, we have 
\[
A_\omega^k x
= A_\omega^k (\bigoplus_{j=1}^{r} \gamma_j v_j)
= \displaystyle \bigoplus_{j=1}^{r} \gamma_j\,(A_\omega^k v_j)
= \displaystyle \bigoplus_{j=1}^{r} \gamma_j\,\beta_j(\omega)^k\,v_j.
\] 
Since each $\beta_j(\omega) \in [0,1]$, the scalar limit $\beta_j(\omega)^k \to 0$ if
$\beta_j(\omega) < 1$, while $\beta_j(\omega)^k = 1$ for all $k$ if $\beta_j(\omega) = 1$.
Thus $A_\omega^k x$ converges, and its limit $\xi_{x,\omega}$ satisfies
$A_\omega \xi_{x,\omega} = \xi_{x,\omega}$, thereby proving the periodicity conclusion as well.
\end{proof}

Note that in the above theorem, if the word $\omega$ contains the letter $i$ at least once, then 
$A_i \xi_{x,\omega} = \xi_{x,\omega}$. In particular, if $\omega$ contains every letter $1,\dots,N$ at least once, then
$\xi_{x,\omega}$ is a common fixed point for the entire collection $\Sigma$, and the limit
$\xi_{x,\omega}$ is independent of the choice of such a word $\omega$. This gives an immediate consequence to a matrix polynomial with coefficients from $\Sigma$. We state this as a corollary and skip its proof.

\begin{corollary}\label{cor:com-eigenvector-polynomial}
Let $\Sigma = \{A_1,\dots,A_N\}\subset M_n(\mathbb{R}_+)$ satisfy the hypotheses of Theorem \ref{thm:JSR-common-eig-periodic}. Also, let $P(\lambda) = B_0 \oplus \lambda B_1\oplus \cdots \oplus 
\lambda^{m-1}B_{m-1}$ be a max--matrix polynomial whose coefficients satisfy $B_j \in \Sigma$ for all
$j = 0,1, \dots, m-1$. Then, any fixed point $\xi_x$ for $A_\omega$, as described in Theorem 
\ref{thm:JSR-common-eig-periodic} is a fixed point of the map $x \mapsto P(1)\otimes x$. In other words,
$P(1)  \xi_x = \xi_x$.
\end{corollary}

We end the paper with the following corollary.

\begin{corollary}\label{cor:JSR-strict-to-zero}
Let $\Sigma = \{A_1,\dots, A_N\} \subset M_n(\mathbb{R}_+)$ be finite and suppose that 
$\rho_\eta(\Sigma) < 1$.  Then for every word $\omega$ and every $x \in \mathbb{R}^n_+$,
$\displaystyle \lim_{k\to\infty} A_\omega^k x = 0$. In particular, the only periodic point of the map 
$x \mapsto A_\omega x$ is $0$.
\end{corollary}

\begin{proof}
Fix a word $\omega$ of length $p$ and choose $r$ such that $\rho_\eta(\Sigma) < r <1$. 
By the definition of $\rho_\eta(\Sigma)$, there exists $K \in \mathbb{N}$ such that for all $t \geq K, 
\displaystyle \sup_{B_1,\dots, B_t \in \Sigma}\eta_{\|\cdot \|}(B_t \cdots B_1) \leq r^t$. For $k$ large enough so that $kp \geq K$, 
the matrix $A_\omega^k$ is a product of length $kp$
with factors in $\Sigma$ and hence $\eta_{\|\cdot \|}(A_\omega^k) \leq r^{kp}$. Therefore 
$\|A_\omega^k\otimes x\| \leq \eta_{\|\cdot \|}(A_\omega^k) \|x\| \leq r^{kp} \|x\| \rightarrow 0$, which implies 
$A_\omega^k x\to 0$ in $\mathbb{R}^n_+$.  If $y \neq 0$ were a periodic
point, say $A_\omega^q y = y$, then $A_\omega^{kq} y = y$ for all $k$, contradicting
the convergence to $0$.  Hence $0$ is the only periodic point.
\end{proof}

We end the paper with an example that illustrates the above Theorem and Corollary.

\begin{example}\label{eg-2}
We construct a family $\Sigma = \{A_1, A_2\} \subset M_4(\mathbb{R}_+)$ illustrating Theorem \ref{thm:JSR-common-eig-periodic}. Let 
$A_1 =
	\begin{pmatrix}
		1   & 0.5 & 0   & 0   \\
		0.3 & 1   & 0   & 0   \\
		0   & 0   & 0.9 & 0.7 \\
		0   & 0   & 0.5 & 0.9
	\end{pmatrix},
\ 
	A_2 =
	\begin{pmatrix}
		0.8 & 1   & 0   & 0   \\
		1   & 0.4 & 0   & 0   \\
		0   & 0   & 0.8 & 0.6 \\
		0   & 0   & 0.4 & 0.8
	\end{pmatrix}, \ v_1 = (1,1,0,0)^T$ and $v_2 = (0,0,1,1)^T$. A direct computation gives
	\[
	A_1 \otimes v_1 = v_1, \quad A_2 \otimes v_1 = v_1,
	\]
and therefore $v_1$ is a common max-eigenvector of $\Sigma$ with corresponding eigenvalue $1$. Similarly,
	\[
	A_1 \otimes v_2 = 0.9\, v_2, \quad A_2 \otimes v_2 = 0.8\, v_2,
	\]
making $v_2$ a common max-eigenvector of $A_1$ with eigenvalue $0.9$, and of $A_2$ with
eigenvalue $0.8$. Thus $E = \{v_1, v_2\}$ is a nontrivial common max-eigenvector set for $\Sigma$. By 
statements $(1)$ and $(2)$ of Theorem \ref{Determining the JSR}, $\rho_\eta(\Sigma) = \mu(A_1 \oplus A_2)$. From the above, we have,
	\[A_1 \oplus A_2  = \begin{pmatrix}
		1 & 1   & 0   & 0   \\
		1   & 1& 0   & 0   \\
		0   & 0   & 0.9 & 0.7 \\
		0   & 0   & 0.5 & 0.9
	\end{pmatrix}.\]
The spectral radius of $A_1 \oplus A_2$ is $\mu(A_1 \oplus A_2) =1$. Thus, $\rho_\eta(\Sigma) = 1$.  Let 
$x = \alpha\, v_1\oplus \beta\, v_2 = (\alpha,\alpha,\beta,\beta)^T \in \mathcal{LC}(E)$, for 
$\alpha, \beta \in \mathbb{R}_+$. Consider the word $\omega = 12$, so that $A_\omega = A_2\otimes A_1$. Then 
$A_\omega \otimes v_1 = v_1$ and $A_\omega \otimes v_2 = (0.9 \times 0.8)\,v_2 = 0.72\,v_2$. By max-linearity,
	\[
	A_\omega^k \otimes x
	\;=\;
	\alpha\, v_1 \oplus \beta\,(0.72)^k\, v_2
	\;\xrightarrow{k\to\infty}\;
	\alpha\, v_1
	\;=\;
	(\alpha,\,\alpha,\,0,\,0)^T
	\;:=\;
	\xi_{x,\omega}.
	\]
Since $A_\omega \otimes \xi_{x,\omega} = \xi_{x,\omega}$, the limit $\xi_{x,\omega}$ is a fixed point of $A_\omega$, confirming Theorem \ref{thm:JSR-common-eig-periodic}. The component along $v_2$ decays geometrically because both eigenvalues ($0.9$ and $0.8$) are strictly less than $1$, while the component along $v_1$ stabilizes because the 
corresponding eigenvalue equals $1$. Moreover, since the word $\omega=12$ contains every letter in $\{1,2\}$, the limit $\xi_{x,\omega}$ is a common fixed point of $\Sigma$ and is independent of the particular choice of such a word.
\end{example}

\section{Concluding Remarks}

We summarize the main results obtained in this paper.

\begin{itemize}
\item Bounded sets of matrix polynomials over max-algebras, their growth in terms of two max-induced seminorms 
and the associated joint spectral radius of the coefficient pool are brought out.
\item The coefficient pool being simultaneously triangularizable is of particular interest and interesting results in this 
setting are brought out.
\item Dynamics of max-linear maps $x \mapsto Ax$ and convergence to periodic points are studied; in particular, 
a global bound in terms of the joint spectral radius of the entire coefficient pool of the bounded set of matrix polynomials, 
the existence of common eigenvectors and convergence to periodic points for all word products simultaneously are brought out.
\end{itemize}

\bibliographystyle{amsplain}

\end{document}